\UseRawInputEncoding
\documentclass[12pt]{article}
\usepackage{amssymb,latexsym}

\usepackage{color}

\usepackage{pdfsync}
\usepackage[colorlinks, 
]{hyperref}
\usepackage{amsmath,amsthm}
\usepackage{indentfirst}
\theoremstyle{plain}
\newtheorem{thm}{Theorem}[section]

\newtheorem{lem}{Lemma}

\theoremstyle{definition}

\numberwithin{equation}{section}
\allowdisplaybreaks

\title{Liouville theorem  of $ VT $ harmonic map from complete noncompact manifold into horoball}

\author{Xiang-Zhi Cao\thanks{School of information engineering, Nanjing Xiaozhuang University, Nanjing {\rm211171}, China.}\thanks{This work is surported by General Project of Natural Science(basic science) Research of 
		Colleges and Universities in Jiangsu province(Grant No.
		22KJD110004)} \thanks{Email:aaa7756kijlp@163.com }}

\begin{document}
	\maketitle

\begin{abstract}
In this paper, we mainly  study Liouville theorem of $ VT $ harmonic map from complete noncompact manifold into  horoball in Cartan-Hardmard manifold. To this aim, we will establish gradient estimates under some  condition on and $ V  $ and  $ T $.
\end{abstract}

\section{Introduction}

 One of  the most important map in geometry is harmonic map which is the generalization of geodesic and submanifold in geometry and plays  an important role in geometry and topology. For the research on existence, uniqueness and regularity of harmonic map between manifold, one can refer to the works(\cite{ES,Eells1978,Schoen1982},\cite{Ham}\cite{sacks}\cite{lin1999gradient}\cite{chang1992finite}\cite{ding1995energy}\cite{hartman1967homotopic})  and reference therein.

 In \cite{jost-yau}, Jost and yau studied the existence of Hermitian harmonic map form compact manifold (with or without boundary) to Riemannian manifold with nonnegative curvature. We  refer the intrested readers to \cite{Grunau2005,jost-yau,li2007hermitian,liu2014hermitian,ni1999hermitian} and reference therein  for the history of Hermitian harmonic map. In \cite{chen2015maximum}, Chen et al. introduced the notions $  V  $ harmonic map . It is obvious that  $ V $ harmonic map is the gerneralization of Hermitian harmonic map.   More  history of research about $ V $ harmonic map  can be found in  \cite{qiu2017heat,vharmonic,Chen2014OmoriYauMP} and reference therein.

Let $ (M,g) $  and  $(N,h)  $ be   Riemannian manifolds, $ u: (M,g)\to (N,h) $.  In 2020, Chen et al.\cite{Chen2020b} introduced the notion of $VT$-harmonic map which  is defined as 
$$\tau(u) + du(V) + {\rm Tr}_g T(du, du) = 0,$$
where $ V\in \Gamma(TM), T\in \Gamma(\otimes^{1,2}(TN)) .$
It is the generalization of $  V $ harmonic map. In   \cite{CaoXiangzhi2018}, we used the methods in \cite{hildebrandt1977existence,hildebrandt1976dirichlet,widman1} to study the existing theorem of VT harmonic from compact manifold with boundary into  manifold, if the tensor $ T $ and the  curvature operator of $ N $ satisfy certain conditons. Later, in \cite{MR4499521}, we obtained the existing theorem of VT harmonic map from compact manifold with boundary into regular ball of certain radius in Riemannian manifold. 
  When $ M $ is complete noncompact, in \cite{Chen2020b}, Chen et al. also derive the existence theorem imposing some conditons on  the tensor $ T $.

In \cite{Chen2020b}, Chen et al. give the following gradient estimates ,
\begin{thm}[c.f. Theorem 9 in \cite{Chen2020b}]\label{l2}
	 Let $\left(M^{m}, g\right)$ be a complete noncompact Riemannian manifold with
	\[
	\operatorname{Ric}_{V}:=\operatorname{Ric}-\frac{1}{2} L_{V} g \geq-A,
	\]
	where $A \geq 0$ is a constant, Ric is the Ricci curvature of $M$ and $L_{V}$ is the Lie derivative. Let $\left(N^{n}, \widetilde{g}\right)$ be a complete Riemannian manifold with sectional curvature bounded from above by a positive constant $\kappa$. Let $u: M \rightarrow N$ be a $V T$-harmonic map such that $u(M) \subset B_{\widetilde{R}}(p)$, where $B_{\widetilde{R}}(p)$ is a regular ball in $N$, i.e., disjoint from the cut locus of $p$ and $\widetilde{R}<\frac{\pi}{2 \sqrt{\kappa}}$. Suppose $\|V\|_{L^{\infty}}<+\infty,\|T\|_{L^{\infty}}<+\infty,\|\nabla T\|_{L^{\infty}}<+\infty$ and
	\begin{equation}\label{l1}
		\begin{split}
		\left(1+(m+1)^{2}-\frac{1}{(m+1)^{2}}\right)\|T\|_{L^{\infty}}^{2}+\frac{\sqrt{\kappa}}{\cos (\sqrt{\kappa} \widetilde{R})}\|T\|_{L^{\infty}}<\frac{\kappa}{\min \{m, n\}} .
		\end{split}
	\end{equation}

	Then, we have
	\[
	|\nabla u| \leq C_{6}(\sqrt{A}+1),
	\]
	where $C_{6}>0$ is a constant depending only on $m, n, \kappa, \tilde{R}, V, T$.
\end{thm}
 Liouville type theroem for $ VT $ harmonic map cann't be obtained from the above  upper bound  . In  \cite{shen1995liouville}, Shen first considered Liouville type theorem of harmonic map $u$  into horoball. In \cite{chen1998harmonic}, Chen obtained Liouville theorem for harmonic map with potential into regular ball or horoball.  In this paper, we give sufficient conditions to derive Liouville theorem for $ VT $ harmonic map into horoball in Cartan-Hardmard manifold. We will give two methods which is adapted from that in  \cite{Chen2020b} \cite{qiu2017heat}\cite{chen1998harmonic}\cite{shen1995liouville}.

\section{ Main theorem and its proofs }
Before proving our main theorems, let us recall same background of horoball in Cartan-Hardmard manifold $ N $. Let $c:[0,+\infty) \rightarrow N$ be a unit speed geodesic. We call $\bar{B}_{c}:=\bigcup_{t>0} \bar{B}_{c(t)}(t)$ the horoball with respect to $c$ centered at $c(+\infty)$. The Busemann function (with respect to $c$ ) is $B(\cdot)=\lim _{t \rightarrow+\infty}\left[t-\mathrm{d}_{N}(c(t), \cdot)\right]$. It is known that $B$ is of $C^{2},|\nabla |=1$ and $B>0$ on $\bar{B}_{c}$ (see \cite{chen1998harmonic}).

\begin{lem}[c.f. Lemma 2.1 in\cite{arxivecao202302}  or \cite{CaoXiangzhi2018}]\label{bochner}
	Let $\left(M^{m}, g\right)$ and $\left(N^{n}, h\right)$ be Riemannian manifolds.  Suppose $ u $ is a $ V T$ harmonic map from $M$ to $N$, then
	\begin{equation}\label{66}
		\begin{split}
			&\frac{1}{2} \left( \Delta_{V} -\frac{\partial }{\partial t}\right) e(u)\\
			\geq& 	(1-\epsilon)|\nabla du|^{2}-\sum_{\beta,\alpha}\bigg(\langle R^{N}(du(e_{\beta}),du(e_\alpha))du(e_{\beta}),du(e_{\alpha})\rangle+\frac{1}{2\epsilon} |T(du(e_{\beta})^{\sharp},du(e_{\alpha}))|^{2}\\ 
			&+\langle(\nabla T)(du(e_{\beta})^{\sharp} ,du(e_{\alpha}),du(e_{\alpha})),du(e_{\beta})\rangle\bigg) +\sum_{\beta}\langle du(\mathrm{Ric}_V(e_{\beta})),du(e_{\beta})\rangle,
		\end{split}
	\end{equation}
	where $ \{e_{\alpha}\}_{\alpha=1}^{m}$ is the orthonormal frame of $ M $ , $\epsilon$  is a arbitrary small  positive constant, $Ric_{V}:=$ Ric $-\frac{1}{2} L_{V} g$, where Ric is the Ricci curvature of $M$ and $L_{V}$ is the Lie derivative.
	
\end{lem}

Now we begin to state our first result and give its proof.

\begin{thm}\label{thm2}
	Let $\left(M^{m}, g\right)$ be a complete noncompact Riemannian manifold with
	\[
	\operatorname{Ric}_{V} \geq -A,
	\]
	where $A \geq 0$ is a constant, Ric is the Ricci curvature of $M$ and $L_{V}$ is the Lie derivative. Let $\left(N^{n}, h\right)$ be a complete Riemannian manifold with sectional curvature bounded from above by $ -\kappa^2 $. Let $u: M\rightarrow N$ be a $V T$-harmonic map such that $u(M) \subset B_c$, where $B_c$ is a horoball in $N$ with respect to some unit speed geodesic $ c $. Suppose $\|V\|_{L^{\infty}}<+\infty$ and	
	\begin{equation*}
		\frac{1}{2\epsilon} |T(X^{\sharp},Y)|^{2}
		+\langle(\nabla T)(X^{\sharp} ,Y,Z),X\rangle \leq 0, \quad \text{and} \quad \|T\|_{\infty}<1.
	\end{equation*} 	
 where $ \epsilon $ is small positive constant and $ X,Y,Z $ is the vector field on $ N $.	Then, we have for any $ \tilde{p} \in M $ and $  a>0 $, we have 
	\[
\sup_{B_{\frac{a}{2}}(\tilde{p})}	|\nabla u| \leq C(\sqrt{A}+\frac{1}{a}+\frac{1}{a^{\frac{1}{2}}}),
	\]
	where $C>0$ is a constant depending only on $m, n, \kappa, V, T$.
\end{thm}

\begin{proof} we use  the method in \cite{shen1995liouville}.  Let $ e=|du |^2 $,  by \eqref{66} and Kato's inequality, we have 
	$$  \frac{1}{2} \Delta_V e  \geq (1-\epsilon)(1+\delta)|\nabla \sqrt{e}|^{2}-K e, $$
	where $\delta=1 / 2 m n$ . Therefore,
\begin{equation}\label{ccc}
	\begin{split}
			|\nabla u| \Delta_V|\nabla u| \geq\left.\left[ (1-\epsilon)(1+\delta)-1\right] |\nabla| \nabla u\right|^{2}-K|\nabla u|^{2}.
	\end{split}
\end{equation}
	Next, let $\phi= \frac{|\nabla u|}{B \circ u} $, where $B$ is the Buseman function on $B_{c}$. We  may assume that  $  B>1 $.  Then we have
	\[
	\Delta_V|\nabla u|=B \circ u \Delta_V\phi+\phi \Delta_V(B \circ u)+2 \nabla(B \circ u) \cdot \nabla \phi.
	\]
	So
	\[
	\Delta_V \phi=\frac{\Delta_V|\nabla u|}{B \circ u}-\frac{\phi \Delta_V(B \circ u)}{B \circ u}-2 \frac{\nabla(B \circ u) \cdot \nabla \phi}{B \circ u}.
	\]
	By  (\ref{ccc}), we get
\begin{equation}\label{equ5}
	\begin{split}
		\Delta_V \phi \geq \varepsilon_1 \frac{\left.|\nabla| \nabla u\right|^{2}}{|\nabla u| B \circ u}-K \phi-\frac{\phi \Delta_V(B \circ u)}{B \circ u}-2 \frac{\nabla(B \circ u) \cdot \nabla \phi}{B \circ u},
	\end{split}
\end{equation}
	where $\varepsilon_1= (1-\epsilon)(1+\delta)-1 $.
	
	We know that by (\cite{shen1995liouville}) 
	\begin{equation*}
		\begin{split}
			-2 \frac{\nabla(B \circ u) \cdot \nabla \phi}{B \circ u}
					\geq-(2-\eta) \frac{\nabla(B \circ u) \cdot \nabla \phi}{B \circ u}-\eta \frac{|\nabla(B \circ u)||\nabla| \nabla u||}{(B \circ u)^{2}}
			+\eta \frac{|\nabla(B \circ u)|^{2}|\nabla u|}{(B \circ u)^{3}}	.
		\end{split}
	\end{equation*}
By setting $2 \varepsilon_1=\eta$, we have
	$$ -\eta \frac{|\nabla(B \circ u)||\nabla| \nabla u||}{(B \circ u)^{2}} \geq-\varepsilon_1 \left(\frac{|\nabla| \nabla u \|^{2}}{(B \circ u)|\nabla u|}+\frac{|\nabla(B \circ u)|^{2}|\nabla u|}{(B \circ u)^{3}}\right). $$
	Therefore, by \eqref{equ5}, we get 
	$$ \Delta_V \phi \geq-(2-2 \varepsilon_1) \frac{\nabla(B \circ u) \cdot \nabla \phi}{B \circ u}+\varepsilon_1 \frac{|\nabla u||\nabla(B \circ u)|^{2}}{(B \circ u)^{3}}	-K\phi-\frac{\phi \Delta_V(B \circ u)}{B \circ u}. $$
	Now for  a fixed point $x_{0} \in M$, we can define a function $F$ on $B_{a}\left(x_{0}\right)$ by
	\[
	F(x)=\left(a^{2}-r^{2}\right) \phi(x)=\left(a^{2}-r^{2}\right) \frac{|\nabla u|}{B \circ u},
	\]
	where $r(x)=\operatorname{dist}_{g}\left(x_{0}, x\right)$. It is easy to see that if $\nabla u \not \equiv 0$, then $F$ must achieve its maximum at some interior point $x^{*}$. We may assume that $r$ is twice differentiable near $x^{*}$. By the maximum principle, we have
	\[
	\nabla F\left(x^{*}\right)=0, \quad  and  \quad 	\Delta_V F\left(x^{*}\right) \leq 0.
	\]
At $x^{*}$, we have 
	\[
	\frac{\nabla r^{2}}{a^{2}-r^{2}}=\frac{\nabla \phi}{\phi},
	\]
\begin{equation}\label{equ6}
	\begin{split}
		-\frac{\Delta_V r^{2}}{a^{2}-r^{2}}+\frac{\Delta_V \phi}{\phi}-\frac{2 \nabla r^{2} \circ \nabla \phi}{\left(a^{2}-r^{2}\right) \phi} \leq 0.
	\end{split}
\end{equation}
We know that (cf. \cite{Chen2020b} )
\[
\Delta_{V} r^{2}=2 r \Delta_{V} r+2|\nabla r|^{2} \leq 2 r\left(A\left(r-r_{0}\right)+\widetilde{C}_{0}\right)+2,
\]
where $r_{0}>0$ is a sufficiently small constant and $\widetilde{C}_{0}:=\max _{\partial B_{r_{0}}(\tilde{p})} \Delta_{V} r$. By \eqref{equ6}, we get
	\begin{equation}\label{ccc1}
		\begin{split}
			0 \geq & \frac{\Delta_V\phi}{\phi}-\frac{2 r\left(A\left(r-r_{0}\right)+\widetilde{C}_{0}\right)+2}{a^{2}-r^{2}}-\frac{8 r^{2}}{\left(a^{2}-r^{2}\right)^{2}} \\
			\geq &-(2-2 \varepsilon) \frac{\nabla(B \circ u) \cdot \nabla \phi}{\phi(B \circ f)}+\varepsilon \frac{|\nabla(B \circ u)|^{2}}{(B \circ u)^{2}}-K-\frac{\Delta_V(B \circ u)}{(B \circ u)} \\
			&-\frac{\left.(2 r\left(A\left(r-r_{0}\right)+\widetilde{C}_{0}\right)+2)\left(a^{2}-r^{2}\right)+8 r^{2}\right)}{\left(a^{2}-r^{2}\right)^{2}}.
		\end{split}
	\end{equation}
	Let
	\[
	A_1=A+\frac{(2 r\left(A\left(r-r_{0}\right)+\widetilde{C}_{0}\right)+2)\left(a^{2}-r^{2}\right)+8 r^{2}}{\left(a^{2}-r^{2}\right)^{2}}.
	\]	
 The inequality (\ref{ccc1}) can be rewritten as
	\begin{equation}\label{equ1}
		\begin{split}
			\bigg[ \varepsilon \frac{|\nabla(B \circ u)|^2}{(B \circ u)^2}-\frac{\Delta_V(B \circ u)}{B \circ u}\bigg]-(2-2 \varepsilon)\frac{2 r}{\left(a^2-r^2\right)} \frac{\nabla(B \circ u) \cdot \nabla r}{B \circ u}-A_1 \leq 0.
		\end{split}
	\end{equation}
Considering that 
$$ \Delta_{V}(B \circ u) \leq 	-2 e(u)+|\nabla B|^2+\|T\|_{\infty} e(u)= (\|T\|_{\infty}-1) e(u) ,$$
 we have 

 \begin{equation*}
	\begin{split}
	\varepsilon \frac{|\nabla(B \circ u)|^2}{(B \circ u)^2}-\frac{\Delta_V(B \circ u)}{B \circ u}\geq \varepsilon \frac{|\nabla(B \circ u)|^2}{(B \circ u)^2}-\epsilon\frac{(\|T\|_{\infty}-1) e(u)}{\left( B \circ u\right)^2 } \geq -\epsilon (\|T\|_{\infty}-2)  \phi^2.
	\end{split}
\end{equation*}
	By the property of Buseman function, we have
$$
\left|\nabla(B \circ u) \cdot \nabla r\right| \leq|\nabla(B \circ u)| \leq|\nabla u| .
$$ Thus \eqref{equ1} gives 
	\begin{equation}\label{}
	\begin{split}
(2-\|T\|_{\infty})\frac{1}{\left(a^2-r^2\right)^2}F^2-(2-2 \varepsilon)\frac{2 r}{\left(a^2-r^2\right)^2} F-A_1 \leq 0.
	\end{split}
\end{equation}
	Also, we may always assume that $B \geq 1$ since Busemann functions are horofunctions.
%
%
%
%
%
	That is also equivalent to
	\[
 (2-\|T\|_{\infty}) F^{2}-2(2-2 \varepsilon) r F-A_1\left(a^{2}-r^{2}\right)^{2} \leq 0.
	\]
	 This gives that 
	\begin{equation*}
		\begin{split}
			F\left(x^{*}\right) & \leq  \frac{2(2-2 \varepsilon) r}{ (2-\|T\|_{\infty})}+\left( \frac{A_1\left(a^{2}-r^{2}\right)^{2}}{ (2-\|T\|_{\infty})}\right)^{\frac{1}{2}}. 
		\end{split}
	\end{equation*}
	Then for any $x_{0} \in M$ and any $a>0$, we have
	\begin{equation*}
		\begin{split}
		\sup _{B_{\frac{a}{2}}\left(x_{0}\right)} \frac{|\nabla u|}{B \circ u} \leq& 	\frac{1}{a^2-r^2}	\bigg[\frac{2(2-2 \varepsilon) r}{ (2-\|T\|_{\infty})}+\left( \frac{A_1\left(a^{2}-r^{2}\right)^{2}}{ (2-\|T\|_{\infty})}\right)^{\frac{1}{2}}\bigg]\\
		\leq& C(m,n,T) \bigg(\frac{1}{a}+\sqrt{A+\frac{(2 r\left(A\left(r-r_{0}\right)+\widetilde{C}_{0}\right)+2)\left(a^{2}-r^{2}\right)+8 r^{2}}{\left(a^{2}-r^{2}\right)^{2}}}\bigg)\\
		&\leq C(m, n,T)\left(\sqrt{A}+\frac{1}{a}+\frac{1}{\sqrt{a}}\right).
		\end{split}
	\end{equation*}
	In particular, we have the following estimate at any point $x_{0} \in M$ ,
	\[
	\left|\nabla u\left(x_{0}\right)\right| \leq C(m, n)\left(\sqrt{A}+\frac{1}{a}+\frac{1}{\sqrt{a}}\right) d\left(u\left(x_{0}\right), c(0)\right),
	\]
	where $a>0$ is any number.
	
\end{proof}
%

Our second result is 

\begin{thm}\label{thm2}
	Let $\left(M^{m}, g\right)$ be a complete noncompact Riemannian manifold with
	\[
	\operatorname{Ric}_{V}:=Ric-\frac{1}{2} L_{V} g \geq-K,
	\]
	where $K \geq 0$ is a constant, Ric is the Ricci curvature of $M$ and $L_{V}$ is the Lie derivative. Let $\left(N^{n}, h\right)$ be a complete Riemannian manifold with sectional curvature bounded from above by $ -a^2 $. Let $u: M\rightarrow N$ be a $V T$-harmonic map such that $u(M) \subset B_c$ and $  e(u)\leq M_1 $, where $B_c$ is a horoball in $N$ with respect to some unit speed geodesic $ c $ . Suppose $\|V\|_{L^{\infty}}<+\infty,$ and 
	if $R^N$ and $ T $ satisfies that 
	\begin{equation*}
	\frac{1}{2\epsilon} |T(X^{\sharp},Y)|^{2}
		+\langle(\nabla T)(X^{\sharp} ,Y,Z),X\rangle \leq 0, \quad \text{and} \quad \|T\|_{\infty}<1.
	\end{equation*}
 where $ \epsilon $ is small positive constant and $ X,Y,Z $ is the vector field on $ N $. 	Then,  for any $ \tilde{p} \in M $, we have 
	\[
	\sup_{B_{\frac{R}{2}}(\tilde{p})} |\nabla u| \leq C_{6}(\sqrt{K}+\frac{1}{R}+\frac{1}{R^{\frac{1}{2}}}),
	\]
	where $C_{6}>0$ is a constant depending only on $m, n, \kappa, \tilde{R}, V, T$.
\end{thm}
\begin{proof} We use the method of \cite{qiu2017heat}\cite{chen1998harmonic}.  Without loss of generality, we assume $ a = 1 $.
	Let $ A=\frac{e(u)}{(1+B(u))^2}, b=1+B(u) $
$$
\Delta_V A=\frac{\Delta_V e}{b^2}-\frac{4 \nabla e \cdot \nabla_V b}{b^3}-\frac{2 e \Delta_V b}{b^3}+\frac{6 e|\nabla b|^2}{b^4},
$$
 this yields
$$
\Delta_V A \geqslant-2K A-\frac{2 A \Delta_V b}{b}+\left[\frac{|\nabla^2 u|^2}{b^2}-\frac{4 \nabla e \cdot \nabla b}{b^3}+\frac{6 e|\nabla b|^2}{b^4}\right] .
$$
The last term on the right-hand side(see  \cite[page 1785]{chen1998harmonic})
$$
\begin{aligned}
	{[\cdots]  }
	\geqslant &-(2-2 \varepsilon) \frac{\nabla A \cdot \nabla b}{b}+\frac{2 \varepsilon A|\nabla b|^2}{b^2},
\end{aligned}
$$
Hence,
\begin{equation}\label{equ2}
	\begin{split}
		\Delta_V A \geqslant-2\left[(m-1) K\right] A-(2-2 \varepsilon) \frac{\nabla A \cdot \nabla b}{b}+\frac{2 \varepsilon A|\nabla b|^2}{b^2}-\frac{2 A \Delta_V b}{b} \text {. }
	\end{split}
\end{equation}
Hessian comparison theorem gives 
$$
\begin{aligned}
	\Delta_V b &=\operatorname{Hess}(B)\left(u_* e_a, u_* e_\alpha\right)+\left\langle\nabla B, \tau_V\left(u^*\right)\right\rangle \\
	& \leqslant-2 e+|\nabla b|^2-\langle\nabla B,\operatorname{Tr}T_g(du,du)\rangle .\\
		& \leqslant-2 e+|\nabla b|^2+\|T\|_{\infty}e(u) .
\end{aligned}
$$
Denote $\mathcal{O}:=-2 e+|\nabla b|^2+\|T\|_{\infty}e(u)$. Then $\Delta b \leqslant \mathcal{O}$. Notice that $ \mathcal{O} \leq 0. $
Substituting this into (\ref{equ2}) and noticing $b>1$, we have
$$
\Delta_V A \geqslant-2 (m-1) K A-(2-2 \varepsilon) \frac{\nabla A \cdot \nabla b}{b}+\left(\frac{2 \varepsilon A|\nabla b|^2}{b^2}-\frac{2 A \mathcal{O}}{b}\right) .$$
Thus, we get 
$$
\frac{2 \varepsilon A|\nabla b|^2}{b^2}-\frac{2 A \mathcal{O}}{b}\geqslant \frac{2 \varepsilon A}{b^2}\left(|\nabla b|^2-\mathcal{O}\right) \geq 2\epsilon A^2(2-\|T\|_{\infty})  .
$$
Hence,
$$
\Delta_V A \geqslant-2(m-1) K A- 2(2-\|T\|_{\infty}) \varepsilon A^2 +(2-2 \varepsilon) \frac{\nabla A \cdot \nabla b}{b}.
$$
 By (2.6) in \cite{qiu2017heat}, we know that 
\begin{equation*}
	\begin{split}
		\Delta_{V} \tilde{\rho}^{2}=2 \tilde{\rho} \Delta_{V} \tilde{\rho}+2|\nabla \tilde{\rho}|^{2} \leq 2 r\left(K\left(\tilde{\rho}-r_{0}\right)+C_1\right)+2,
	\end{split}
\end{equation*}
  where $r_{0}>0$ is a sufficiently small constant and $C_1:=\max _{\partial B_{r_{0}}(\tilde{p})} \Delta_{V} r.$

  As in \cite{qiu2017heat},  	let $\psi(t) \in C^{2}([0,+\infty))$ such that

  	\begin{equation}	
  			\psi(t)=
  			\begin{cases}	
  					1&t \in\left[0, \frac{1}{2}\right]\\
  					0& t \in[1,+\infty)	
  				\end{cases}	
  		\end{equation}

  	$$ \psi(t) \in[0,1], \psi^{\prime}(t) \leq 0, \psi^{\prime \prime}(t) \geq-C_{3}   ,  \frac{\left|\psi^{\prime}(t)\right|^{2}}{\psi(t)} \leq C_{3}, $$  	
  	where $C_{3}$ is a positive constant. Set $\phi(x)=\psi\left(\frac{\tilde{\rho}(x)}{R}\right)$; then it is easy to see that (cf. \cite[ page 2275]{qiu2017heat} ) 
  	
  	$$ \Delta_{V} \phi=\frac{\psi^{\prime} \Delta_{V} \tilde{\rho}}{R}+\frac{\psi^{\prime \prime}|\nabla \tilde{\rho}|^{2}}{R^{2}} \geq -\frac{\sqrt{C_{3}}\left(A\left(R-r_{0}\right)+C_{1}\right)}{R}-\frac{C_{3}}{R^{2}} . $$


%
%


	Let $G(x)=\phi(x) A $. 
	If $G(x, t)$ achieves its maximum at $x_{1} \in B_{R}(\tilde{p})$ , then without loss of generality, we assume that $G\left(x_{1}\right)>0$; By the maximum principle, at $x_{1}$ we have
	
	\[ \nabla A=-\frac{\nabla \phi}{\phi} A, \quad \Delta_{V}(\phi A) \leq 0.  \]
	Therefore, at $x_{1}$, we have 
	\begin{align*}
			0 \geq &\left(\Delta_{V}\right)(\phi A)=A \Delta_{V} \phi+2\langle\nabla \phi, \nabla A\rangle+\phi\Delta_{V}A \\
			\geq &-\left(\frac{\sqrt{C_{3}}\left(K\left(R-r_{0}\right)+C_{1}\right)}{R}+\frac{C_{3}}{R^{2}}\right) A-2\frac{|\nabla \phi|^{2}}{\phi} A \\
			&+\phi\left(-2(m-1) K A+2(2-\|T\|_{\infty}) \varepsilon A^2 -(2-2 \varepsilon) \frac{\nabla A \cdot \nabla b}{b}\right) .
		\end{align*}	
Hence, we have 
\begin{equation*}
	\begin{split}
		0 \geqslant 2(2-\|T\|_{\infty}) \epsilon A^{2} \phi-2\left(CK+\frac{C_{7}}{R^{2}}+\frac{C_{7} \sqrt{K}}{R}\right) A \phi-C_{9} \frac{A \sqrt{A \phi}}{R},
	\end{split}
\end{equation*}	
	The quadratic formula implies that
	
	\[
	\sup_{B_{\frac{R}{2}}(\tilde{p})}|\nabla u| \leq C_{6}(\sqrt{K}+\frac{1}{R}+\frac{1}{R^{\frac{1}{2}}}),
	\]
	where $C_{6}>0$ is a constant depending only on $m, n, \kappa, \tilde{R}, V, T$.

\end{proof}

\end{document}